\documentclass[12pt]{amsart}

\usepackage{fullpage}
\usepackage{amsmath}
\usepackage{amsfonts}
\usepackage{amssymb}
\usepackage{amsthm}
\usepackage[all,knot,poly]{xy}

\usepackage[all]{xy}
\theoremstyle{plain}
\usepackage{latexsym}

\usepackage{enumerate}

\newtheorem{theorem}{Theorem}[section]
\newtheorem{lemma}[theorem]{Lemma}
\newtheorem{proposition}[theorem]{Proposition}

\newtheorem{remark}[theorem]{Remark}
\newtheorem{definition}[theorem]{Definition}

\usepackage{color}

\def\cat{\protect\operatorname{cat}}
\def\secat{\protect\operatorname{secat}}

\def\cl{\protect\operatorname{cl}}
\def\TC{\protect\operatorname{TC}}
\def\F{\protect\operatorname{Conf}}
\def\cl{\protect\operatorname{cl}}

\def\R{\protect\mathbb{R}}
\def\co{\colon\thinspace}

\title{Sequential motion planning of \\non-colliding particles in euclidean spaces}
\author{Jes\'us Gonz\'alez\textsuperscript{1}}
\author{Mark Grant\textsuperscript{2}}
\thanks{\textsuperscript{1}~~Supported by Conacyt Research Grant 221221.}
\thanks{\textsuperscript{2}~~Corresponding author.}

\address{Departamento de Matem\'aticas, Centro de Investigaci\'on y de Estudios Avanzados del IPN, Av.~IPN 2508, Zacatenco, M\'exico City 07000, M\'exico}
\email{jesus@math.cinvestav.mx}

\address{School of Mathematics \& Statistics, Newcastle University, Herschel Building, Newcastle upon Tyne NE1 7RU, UK}
\email{mark.grant@newcastle.ac.uk}


\begin{document}
\begin{abstract}
In terms of Rudyak's generalization of Farber's topological complexity of the path motion planning problem in robotics, we give a complete description of the topological instabilities in any sequential motion planning algorithm for a system consisting of non-colliding autonomous entities performing tasks in space whilst avoiding collisions with several moving obstacles. The Isotopy Extension Theorem from manifold topology implies, somewhat surprisingly, that the complexity of this problem coincides with the complexity of the corresponding problem in which the obstacles are stationary.\end{abstract}
\maketitle
\tableofcontents

\section{Statement of results}
We consider the sequential motion planning problem for $n$ objects moving in $\R^m$, avoiding collisions with each other and with $p$ moving obstacles, the trajectories of which are prescribed in advance.

\medskip
We begin by fixing some notation. For a $k$-tuple $a=(a_1,\ldots,a_k)\in\mathbb{R}^k$, let $\overline{a}$ denote the underlying set $\{a_1,\ldots,a_k\}$. Let the trajectories of $p$ moving particles in $\mathbb{R}^m$---the `obstacles'---be given by continuous maps $$q_i:I=[0,1]\to \R^m, \mbox{  \ } i=1,\ldots ,p.$$ We assume that the obstacles do not collide with each other, so that $i\neq j$ implies that $q_i(t)\neq q_j(t)$ for all $t\in I$. Taken together these trajectories form a map
$$
\mathcal{Q}=(q_1, \ldots , q_p): I\to \F(\R^m,p)
$$
where $\F(X,p)$ stands for the usual configuration space of $p$-tuples of pairwise distinct points in a given space $X$. We are interested in the motion planning problem which takes as input a sequence of configurations $A_1,A_2,\ldots , A_s$ with each $$A_i \in \F\left(\mathbb{R}^m-\overline{\mathcal{Q}\left(\frac{i-1}{s-1}\right)},n\right),$$ and outputs a path $\gamma\co I\to \F(\R^m,n)$ such that
\begin{equation}\label{avoids}
\overline{\gamma(t)}\cap\overline{\mathcal{Q}(t)}=\emptyset\qquad\mbox{for all values of }t\in I,
\end{equation}
and
\begin{equation}\label{visits}
\gamma\left(\frac{i-1}{s-1}\right) = A_i\qquad \mbox{for }i=1,\ldots , s.
\end{equation}

The above setting models mathematically the problem of finding trajectories for $n$ objects in $\R^m$ (here represented as points), which navigate from an initial configuration to a final configuration and visit $s-2$ intermediate configurations in sequence, whilst simultaneously avoiding collisions with each other and with $p$ moving obstacles (also represented by points in $\R^m$). Problems of this type arise naturally in practice, for instance in air traffic control or in factory assembling cycles. In such situations, our results become relevant (indeed critical) when the number of particles to be controlled becomes large, so that efficient motion planners have a clear advantage over on-line decision heuristics.

\medskip
We next indicate how solutions of this motion planning problem correspond to sections of a particular fibration. Let
$$
E_n(\mathcal{Q})=\left\{ \gamma\co I\to \F(\R^m,n)\mid \overline{\gamma(t)}\cap\overline{\mathcal{Q}(t)}=\emptyset
\mbox{ for all }t\in I \right\},
$$
topologised as a subset of the path space $\F(\R^m,n)^I$, and let
$$
B_{n,s}(\mathcal{Q})=\left\{(A_1,\ldots , A_s)\in \F(\R^m,n)^{\times s} \mid\overline{A_i}\cap\overline{\mathcal{Q}\left(\frac{i-1}{s-1}\right)}=\emptyset\mbox{ for }i=1,\ldots , s\right\},
$$
topologised as a subset of the $s$-fold cartesian product $\F(\R^m,n)^{\times s}$. There is an evaluation map (which will be shown below to be a fibration)
\begin{equation}\label{mmo}
\pi_{n,s}(\mathcal{Q})\co E_n(\mathcal{Q})\to B_{n,s}(\mathcal{Q})
\end{equation}
sending a path $\gamma$ to the $s$-tuple $\big( \gamma(0),\gamma(\frac{1}{s-1}),\ldots , \gamma(\frac{s-2}{s-1}),\gamma(1)\big)$. Note that a (possibly non-continuous) section of this map corresponds to an algorithm solving our motion planning problem. Asking for the minimal number of continuity instabilities (in the sense of~\cite{MR2074919}) among such algorithms leads to the following:

\begin{definition}{\em
The complexity of the $s$-sequential motion planning problem for $n$ objects moving in $\R^m$, avoiding collisions with each other and with $p$ moving obstacles parametrized by $\mathcal{Q}$, is given by the sectional category\footnote{We work with the reduced version of sectional category.} $\secat(\pi_{n,s}(\mathcal{Q}))$ of the fibration $\pi_{n,s}(\mathcal{Q})$.
}\end{definition}

As a first step toward the determination of the above invariant, we generalize \cite[Theorem~3.1]{MR2359030} by showing that $\secat(\pi_{n,s}(\mathcal{Q}))$ is independent of the actual trajectory $\mathcal{Q}$ of the obstacles. In particular it is enough to consider the case when the $p$ obstacles remain stationary.

\begin{theorem}\label{moving}
The map $\pi_{n,s}(\mathcal{Q})\co E_n(\mathcal{Q})\to B_{n,s}(\mathcal{Q})$ is a fibration, whose fibrewise homeomorphism type is independent of the trajectories of the moving obstacles. Indeed
\begin{equation}\label{movest}
\secat\left(\pi_{n,s}(\mathcal{Q})\right) = \TC_s\left(\F(\R^m-Q_p,n)\right).
\end{equation}
Here $Q_p=\left\{ q_1(0),\ldots , q_p(0)\right\}\subset\R^m$, and $\TC_s(X)$ stands for Rudyak's higher topological complexity of a space $X$ introduced and studied in~\cite{bgrt,MR2593704}.\end{theorem}

The actual value of the right-hand side in~(\ref{movest}) is given as follows:

\begin{theorem}\label{princi1}
Let $m,n,p,s$ be nonnegative integers with $n\geq1$, $m,s\geq2$, and such that $n\ge2$ if $p=0$. Then
\begin{equation}\label{global}
\TC_s(\F(\mathbb{R}^m-Q_p,n))=\begin{cases}
s(n-1)-1, & \mbox{if $p=0$ and $m\equiv0\bmod2$;}\\
s(n-1), & \mbox{if $p=0$ and $m\equiv1\bmod2$;}\\
sn-1, & \mbox{if $p=1$ and $m\equiv0\bmod2$;}\\
sn, & \mbox{otherwise.}
\end{cases}
\end{equation}
\end{theorem}

The assumption that rules out the case $(n,p)=(1,0)$ is just meant to simplify the expression on the right-hand side of~(\ref{global}), as it avoids the case of the contractible space $\F(\mathbb{R}^m,1)$.

\medskip
Theorems~\ref{moving} and~\ref{princi1} generalize results in~\cite{MR2470845,MR2359030,FY}; our method of proof follows those used by Farber, Grant, and Yuzvinsky.

\section{Rudyak's higher TC and proof of Theorem~\ref{moving}}\label{s-introd}
Recall from~\cite{MR2593704} that the {\it $s$-th topological complexity} of a path-connected space $X$, $\TC_s(X)$, is the sectional category of the fibration
$$
e_s=e_s^X:X^{J_s}\rightarrow X^s,\quad e_s(\gamma)=(\gamma({1_1}),\ldots ,\gamma({1_s}))
$$
where $J_s$ is the wedge of $s$ closed intervals $[0,1]$ (each with $0\in[0,1]$ as the base point), and $1_i$ stands for $1$ in the $i^{\mathrm{th}}$ interval. Equivalently, $\TC_s(X)$ can be defined as the sectional category of the evaluation map $$\pi_s=\pi_s^X\colon X^{[0,1]}\to X^s,\quad \pi_s(\gamma)=\left(\gamma(0),\gamma\left(\frac{1}{s-1}\right),\gamma\left(\frac{2}{s-1}\right),\ldots ,\gamma\left(\frac{s-2}{s-1}\right),\gamma(1)\right).
$$
We refer the reader to~\cite{bgrt} for basic properties of the $s$-th topological complexity.

\begin{proof}[Proof of Theorem~$\ref{moving}$]
It suffices to show that there are homeomorphisms $F$ and $G$ rendering a commutative diagram
\begin{equation}\label{fibre}
\xymatrix{
E_n(\mathcal{Q}) \ar[r]^-{F} \ar[d]^{\pi_{n,s}(\mathcal{Q})} & \F(\R^m-Q_p,n)^{[0,1]} \ar[d]^{\pi_s} \\
B_{n,s}(\mathcal{Q}) \ar[r]^-{G} & \F(\R^m-Q_p,n)^{\times s}.
}
\end{equation}
By the topological Isotopy Extension Theorem (see \cite[Corollary 1.4]{MR0283802}, for instance) there exists an ambient isotopy of $\R^m$ extending the isotopy of $p$ points $\mathcal{Q}(t)$. Explicitly, there exists a family of homeomorphisms $\varphi_t\co\R^m\to \R^m$, varying continuously with $t\in I$, such that
\begin{enumerate}
\item $\varphi_0=\mathrm{id}\co\R^m\to\R^m$, and
\item $\varphi_t(q_i(t))=q_i(0)$ for all $t\in I$ and $i=1,\ldots, p$.
\end{enumerate}
The map $G$ is then defined by
$$
G(A_1,\ldots , A_s)=\left(\varphi_0(A_1),\ldots ,\varphi_{\frac{i-1}{s-1}}(A_i), \ldots , \varphi_1(A_s)\right)
$$
for $(A_1,\ldots , A_s)\in B_{n,s}(\mathcal{Q})$. Likewise, for $\gamma\in E_n(\mathcal{Q})$ and $t\in[0,1]$, $F(\gamma)(t)$ is the $n$-tuple obtained by applying $\varphi_t$ to each coordinate of $\gamma(t)$. The maps $F$ and $G$ are clearly homeomorphisms, which make the diagram (\ref{fibre}) commute, thus completing the proof.
\end{proof}

\section{Homotopy obstructions for multisectioning a fibration}\label{genus}
For a fibration $p:E\to B$ with fiber $F$, let $p(\ell):E(\ell)\to B$ be the ($\ell+1$)-th fiberwise join power of~$p$. This is a fibration with fiber $F^{\star(\ell+1)}$, the $(\ell+1)$-iterated join of $F$ with itself. It is well known that, if $B$ is paracompact, a necessary and sufficient condition for having $\secat(p)\leq \ell$ is that $p(\ell)$ admits a global section. Thus, the following result---a direct generalization of~\cite[Theorem~1]{Sch58}---gives a useful cohomological identification of the first obstruction for multi-sectioning $p$.

\begin{theorem}\label{classabstr}
Let $p:E\to B$ be a fibration with fiber $F$ whose base $B$ is a CW complex. Assume $p$ admits a section $\phi$ over the $k$-skeleton $B^{(k)}$ of $B$ for some $k\geq1$. If $F$ is $k$-simple and the obstruction cocycle to the extension of $\phi$ to $B^{(k+1)}$ lies in the cohomology class
\begin{equation}\label{eta}
\eta\in H^{k+1}(B;\{\pi_k(F)\}),
\end{equation}
then $p(\ell)$ admits a section over $B^{(k+1)(\ell+1)-1}$ whose obstruction cocycle to extending to $B^{(k+1)(\ell+1)}$ belongs to the cohomology class
\begin{equation}\label{potencia}
\eta^{\ell+1}\in H^{(k+1)(\ell+1)}(B;\{\pi_{k\ell+k+\ell}(F^{\star(\ell+1)})\}).
\end{equation}
Here $\eta^{\ell+1}$ denotes the image of the $(\ell+1)$-fold cup power of $\eta$ under the $\pi_1(B)$-homomorphism of coefficients $\pi_k(F)^{\otimes (\ell+1)}\to \pi_{k\ell+k+\ell}(F^{\star(\ell+1)})$ given by iterated join of homotopy classes.
\end{theorem}

\begin{remark}\label{kernel}{\em
Assume in the theorem that $F$ is $(k-1)$-connected. Since $\eta$ depends only on $p$, and since the pull-back $p^*(p)$ admits a tautological section, we have $p^*(\eta)=0$ \emph{a fortiori}.
}\end{remark}

\section{Higher TC of Euclidean configuration spaces}\label{case-of-TC}

This section's goal is to prove the $p=0$ case of Theorem~\ref{princi1}, namely:
\begin{theorem}\label{htcofc}
Let $n,m,s\geq2$. The $s$-th higher topological complexity of  the configuration space $\F(\mathbb{R}^m,n)$ of ordered $n$-tuples on the $m$-dimensional Euclidean space is given by
\begin{equation}\label{respuno}
\TC_s(\F(\mathbb{R}^m,n))=\begin{cases}
s(n-1)-1, & m \mbox{ even;}\\
s(n-1), & m \mbox{ odd.}
\end{cases}
\end{equation}
\end{theorem}
For $s=2$ this specializes to the main result in~\cite{MR2470845}. On the other hand, for $n=2$ this recovers the calculation in~\cite{bgrt} of the higher topological complexity of spheres. It should be possible to adapt the calculations in this paper (say under the ``non-broken-circuit'' viewpoint of~\cite{FY}) to study the higher topological complexity of complements of (suitably nice) complex hyperplane arrangements.

\medskip
The upper bound $\TC_s(\F(\mathbb{R}^m,n))\leq s(n-1)$ is a consequence of the well-known inequality $ \cat(X\times Y)\leq  \cat(X)+ \cat(Y)$, and the easy facts that  $\TC_s(X)\leq  \cat(X^s)$ and $ \cat(F(\mathbb{R}^m,n))=n-1$ (the latter is observed in~\cite{MR2508218}). Alternatively one can use~\cite[Theorem~3.9]{bgrt} since  $\F(\mathbb{R}^m,n)$ is an $(m-2)$-connected space with the homotopy type of a CW complex of dimension $(n-1)(m-1)$. On the other hand, the fact that the right-hand side of~(\ref{respuno}) is a lower bound for $\TC_s(\F(\mathbb{R}^m,n))$ follows from~\cite[Theorem~3.9]{bgrt} and the description below of $\cl_s(\F(\mathbb{R}^m,n))$, the cup length of elements in the kernel of the map induced in cohomology by the iterated (thin) diagonal
\begin{equation}\label{thin}
\F(\mathbb{R}^m,n)\to \F(\mathbb{R}^m,n)^s.
\end{equation}

\begin{proposition}\label{tcl}
Let $n,m,s\geq2$ and take $\delta_m\in\{0,1\}$ with $\delta_m\equiv m\bmod2$. Then $$\cl_s(\F(\mathbb{R}^m, n))= s(n-1)-1+\delta_m.$$
\end{proposition}

In preparation for the proof of Proposition~$\ref{tcl}$, recall from~\cite{FH,OT} that the cohomology ring $H^*(\F(\mathbb{R}^m, n))$ is generated by elements $A_{ij}\in H^{m-1}(\F(\mathbb{R}^m, n))$ for $1\leq j<i\leq n$ subject only to the relations
\begin{eqnarray}
A^2_{ij}&=&0\quad\mbox{and}\label{exterior}\\
A_{ik} A_{ij}&=&(A_{ik}-A_{ij})A_{kj}\quad\mbox{for}\quad i>k>j.\label{YBrelations}
\end{eqnarray}
In particular, the monomials $A_{i_1j_1}\cdots A_{i_rj_r}$ with $i_u\neq i_v$ for $u\neq v$ form an additive basis. Order of factors will not be an issue  as it suffices to work with $\mathbb{Z}_2$ coefficients when $m$ is even---but we will have to use $\mathbb{Z}$-coefficients for an odd $m$.

\begin{proof}[Proof of Proposition~$\ref{tcl}$]
Let $w^{(\ell)}$ denote the pull-back of a cohomology class $w$ under map $F(\mathbb{R}^m,n)^s \to F(\mathbb{R}^m,n)$ projecting onto the $\ell$-th cartesian coordinate. The element
\begin{equation}\label{longitudmaxima}
\pi =\prod^n_{i=2} \left(A^{(1)}_{i1}+A^{(2)}_{i1}+\cdots +A^{(s-1)}_{i1}-(s-1) A^{(s)}_{i1}\right)^s
\end{equation}
is a product of $s(n-1)$ factors, all of which clearly lie in the kernel of the iterated diagonal $\F(\mathbb{R}^m,n)\to \F(\mathbb{R}^m,n)^s$. Therefore, the equality $\cl_s (\F(\mathbb{R}^{odd}, n))= s(n-1)$ follows from the considerations in the paragraph previous to Proposition~\ref{tcl} together with the next computation giving the non-triviality of $\pi$ for odd $m$.

\medskip
In view of~(\ref{exterior}), for $i=2, \ldots , n$ we have
\begin{eqnarray*}
\lefteqn{\!\!\!\!\!\!\!\!\!\!\!\!\!\!\!\!\!\!\left(A_{i1}^{(1)} +A^{(2)}_{i1} +\cdots + A^{(s-1)}_{i1}-(s-1)A^{(s)}_{i1}\right)^s=}
\\
&=& sA^{(1)}_{i1} \left(A^{(2)}_{i1}+\cdots + A^{(s-1)}_{i1}-(s-1) A^{(s)}_{i1}\right)^{s-1}\\
&=&sA^{(1)}_{i1} \left((s-1) A^{(2)}_{i1}\right)\left(A^{(3)}_{i1}+\cdots + A^{(s-1)}_{i1}-(s-1) A^{(s)}_{i1}\right)^{s-2}\\
&=& \cdots\\
&=& sA^{(1)}_{i1} \left((s-1) A^{(2)}_{i1}\right)\cdots\left(3A^{(s-2)}_{i1}\right)\left(A^{(s-1)}_{i1}-(s-1) A^{(s)}_{i1}\right)^2\\
&=& s! (1-s)A^{(1)}_{i1} A^{(2)}_{i1}\cdots A^{(s-1)}_{i1} A^{(s)}_{i1}.
\end{eqnarray*}
So
$$
\pi = \prod^n_{i=2} \left(s!(1-s) A^{(1)}_{i1} A^{(2)}_{i1}\cdots A^{(s)}_{i1}\right)
= \left(s! (1-s)\rule{0mm}{3.8mm}\right)^{n-1} \mu^{(1)}\mu^{(2)}\cdots \mu^{(s)} \neq 0
$$
since $\mu=A_{21}A_{31}\cdots A_{n1} \neq 0$ (indeed, $\mu$ is a basis element).

\medskip
The rest of the proof focuses on the case when $m$ is even and, for convenience, is dealt with in the next independent result.
\end{proof}

\begin{lemma}\label{indpdt}
Let $m$ be an even positive integer. Then:
\begin{enumerate}[(i)]
\item\label{g1} $\cl_s(\F(\mathbb{R}^{m}, n))<s(n-1)$.
\item\label{g2} The element
$$
\mu_s=\left(\prod\left(A_{i1}^{(1)}-A_{i1}^{(\ell)}\right)\rule{0mm}{6.8mm}\right)\left(\rule{0mm}{6.8mm}\prod_{i=3}^n\left(A_{i2}^{(1)}-A_{i2}^{(2)}\right)\right),
$$
where the first product is taken over all pairs $(i,\ell)$ with $2\leq i\leq n$ and $2\leq\ell\leq s$, is a non-zero product of $s(n-1)-1$ factors, all of which lie in the kernel of the morphism induced by the iterated diagonal map $\F(\mathbb{R}^m,n)\to\F(\mathbb{R}^m,n)^s$.
\end{enumerate}
\end{lemma}
\begin{proof}
The case of~(\ref{g1}) is easy although a bit short-circuited (just as is the $s=2$ case analyzed in~\cite{FY}): The key point comes from the homeomorphim $\F(\mathbb{C},n)\cong\mathbb{C}^*\times\F(\mathbb{C},n)/\mathbb{C}^*$ (cf.~\cite[Proposition~5.1]{OT}). Here $\F(\mathbb{C},n)/\mathbb{C}^*\subset\mathbb{C}\mathrm{P}^{n-1}$ which, as explained in the proof of Theorem~6 in~\cite{FY}, can be identified with a (not-necessarily central) complex arrangement of rank $n-2$, so that $\F(\mathbb{C},n)/\mathbb{C}^*$ has the homotopy type of a CW~complex of dimension $n-2$. Thus the subadditivity of $\TC_s$ (\cite[Proposition~3.11]{bgrt}) yields $\TC_s(\F(\mathbb{C},n))<s(n-1)$. This of course implies~(\ref{g1}) for $m=2$; the general case of an even $m$ follows by noticing that the cohomology ring of $\F(\mathbb{C},n)$ differs from that for $\F(\mathbb{R}^{even},n)$ only by a `grading homothety'.

\smallskip
Settling~(\ref{g2}) requires a cohomological calculation which, although similar, is slightly less direct than the one handling the non-triviality of the element in~(\ref{longitudmaxima}). In order to simplify matters, we note that it suffices to do the calculation with $\mathbb{Z}_2$-coefficients, where signs can safely be ignored, and that it is enough to show the non-triviality of
\begin{equation}\label{elew}
w_s=A_{21}^{(1)}\mu_s.
\end{equation}
We proceed by induction on $s$, noticing that the grounding case $s=2$ is done by Farber and Yuzvinsky (see the case of reflection arrangements for reflection groups of types $A_n$ at the end of~\cite[Section~3]{FY}). In detail, using the mod 2 analogue of~(\ref{exterior})  we get
\begin{eqnarray}
w_2&=&A_{21}^{(1)}\left(\,\prod_{i=2}^n\left(A_{i1}^{(1)}+A^{(2)}_{i1}\right)\right)\left(\,\prod_{i=3}^n\left(A_{i2}^{(1)}+A^{(2)}_{i2}\right)\right)\nonumber\\
&=&A_{21}^{(1)}A_{21}^{(2)}\prod_{i=3}^{n}\left[\left(
A^{(1)}_{i1}+A^{(2)}_{i1}
\right)\left(
A^{(1)}_{i2}+A^{(2)}_{i2}
\right)\right].\label{auxiliar}
\end{eqnarray}
In view of~(\ref{YBrelations}), $A^{(r)}_{i1}A^{(r)}_{i2}$ is divisible by $A^{(r)}_{21}$ ($r=1,2$), so that~(\ref{auxiliar}) reduces to
\begin{equation}\label{todos}
w_2=A_{21}^{(1)}A_{21}^{(2)}\prod_{i=3}^{n}\left(
A^{(1)}_{i1}A^{(2)}_{i2}
+A^{(2)}_{i1}A^{(1)}_{i2}
\right)=\sum\left(\prod_{i=2}^nA^{(1)}_{i\,j_1(i)}\prod_{i=2}^nA^{(2)}_{i\,j_2(i)}\right)
\end{equation}
where $j_1(2)=j_2(2)=1$, and $\{j_1(i),j_2(i)\}=\{1,2\}$ for $i\geq3$. Note that the expression on the right-hand side of~(\ref{todos}) is non-zero (grounding the induction) since it is in fact a sum of $2^{n-2}$ different basis elements. It also follows that
\begin{equation}\label{porcierto}
w_2 A^{(1)}_{i1}=0
\end{equation}
for any $i$. As for the inductive step,
\begin{equation}\label{elewtmas1}
w_{s+1}=w_s\prod_{i=2}^n\left(A^{(1)}_{i1}+A^{(s+1)}_{i1}\right)=w_s\hspace{.4mm}A^{(s+1)}_{21}A^{(s+1)}_{31}\cdots A^{(s+1)}_{n1}
\end{equation}
for $s\geq2$, where the last equality follows from~(\ref{porcierto}) and the fact that $w_2$ divides $w_t$. But, as an element of $$H^*\!\left(\F(\mathbb{R}^m,n)^{s+1};\mathbb{Z}_2\right)=H^*(\F(\mathbb{R}^m,n)^s;\mathbb{Z}_2)\otimes H^*(\F(\mathbb{R}^m,n);\mathbb{Z}_2),$$
the element on the right hand side of~(\ref{elewtmas1}) is non-trivial since, by induction, $w_s\neq0$ as an element of $H^*(\F(\mathbb{R}^m,n)^s;\mathbb{Z}_2)$.
\end{proof}
\begin{proof}[Proof of Theorem~$\ref{htcofc}$]
It only remains to prove
\begin{equation}\label{elbuenoT}
\TC_s(\F(\mathbb{R}^{m},n))<s(n-1)
\end{equation}
for $m$ even. We can assume $m\geq4$, in view of the proof of part (i) in Lemma~\ref{indpdt}. Note the single obstruction to~(\ref{elbuenoT}) lies in
$$
H^{s(n-1)(m-1)}\left(\F(\mathbb{R}^m,n)^s\,;\,\pi_{s(n-1)(m-1)-1}\left(\left(\Omega\F(\mathbb{R}^m,n)^{s-1}\right)^{\star(s(n-1))}\right)\rule{0mm}{6mm}\right).
$$
In order to get a hold on this obstruction, we use Theorem~\ref{classabstr}. Since the fiber of $$p:=e_s:\F(\mathbb{R}^m,n)^{J_s}\to\F(\mathbb{R}^m,n)^s,$$ $\Omega\F(\mathbb{R}^m,n)^{s-1}$, is $(m-3)$-connected, there are no obstructions for picking a section $\phi$ over the $(m-2)$-skeleton of $\F(\mathbb{R}^m,n)^s$ (so $k:=m-2$ in Theorem~\ref{classabstr}). As noted in Remark~\ref{kernel}, the corresponding class~(\ref{eta}) containing the obstruction to the extension of $\phi$ to the $(m-1)$-skeleton does not depend on the chosen $\phi$, and lies in the kernel of the morphism induced by $e_s$, i.e.~in the kernel of the morphism induced by the iterated diagonal~(\ref{thin}). Taking $\ell:=s(n-1)-1$ in Theorem~\ref{classabstr}, we get a section of $e_s(s(n-1)-1)$ over the $(s(n-1)(m-1)-1)$-skeleton of $\F(\mathbb{R}^m,n)^s$ whose obstruction to extending to the $(s(n-1)(m-1))$-skeleton lies in the corresponding class~(\ref{potencia}). Since $\F(\mathbb{R}^m,n)^s$ has the homotopy type of a CW complex of dimension $s(n-1)(m-1)$, the proof is complete in view of item~(i) in Lemma~\ref{indpdt}, which gives the the triviality of the aforementioned class~(\ref{potencia}).
\end{proof}

\section{Stationary obstacles}
This section deals with the proof of Theorem~$\ref{princi1}$ for $p\geq1$.
The case $p=1$ follows from Theorem~\ref{htcofc} and the fact that the fiber inclusion in the fibration
$$
\F(\mathbb{R}^m-Q_1,n)\to\F(\mathbb{R}^m,n+1)\to\mathbb{R}^m
$$
is a homotopy equivalence. On the other hand, the case $n=1$ follows from~\cite[Corollary~2]{glo}:
$$
\TC_s(\F(\mathbb{R}^m-Q_p,1))=\TC_s(\vee_p S^{m-1})=\begin{cases}
s-1, & \mbox{if $p=1$ and $m$ even;} \\ s, & \mbox{otherwise.}
\end{cases}
$$
Thus, we focus in this section on the $n,p\geq2$ case of Theorem~\ref{princi1}, namely:

\begin{theorem}\label{np2}
For $m,n,p,s\geq2$, $\TC_s(\F(\mathbb{R}^m-Q_p,n))=sn$.
\end{theorem}

The case $m\leq3$ and $s=2$ in Theorem~\ref{np2} is~\cite[Theorems~5.1 and~6.1]{MR2359030}.

\medskip
In preparation for the proof of Theorem~\ref{np2}, we start by recalling the multiplicative structure of the cohomology  of $\F(\mathbb{R}^m-Q_p,n)$ for any ring of coefficients. The following facts can be found in~\cite{FH}.

\medskip
Consider the fibration
$$
\F(\mathbb{R}^m-Q_p,n)\stackrel{\iota}\to\F(\mathbb{R}^m,p+n)\stackrel{\pi}\to\F(\mathbb{R}^m,p)
$$
where $\pi$ projects a $p+n$ tuple to its first $p$ coordinates. The corresponding Serre spectral sequence has a trivial system of local coefficients, and collapses from its second term. In particular, $\iota^*$ is surjective and its kernel is generated by the elements of degree one in the image of $\pi^*$. An additive basis for $H^*(\F(\mathbb{R}^m-Q_p,n))$ is then given by the ($\iota^*$-images of the) monomials $A_{i_1 j_1}\cdots A_{i_\ell j_\ell}$ in $H^*(\F(\mathbb{R}^m,p+n))$ satisfying $p+n\geq i_1>\cdots>i_\ell\geq p+1$. Note that~(\ref{exterior}) and~(\ref{YBrelations}) give the relation
\begin{equation}\label{relazero}
A_{i,k}A_{i,j}=0\quad\mbox{for}\quad j,k\leq p.
\end{equation}
in $H^*(\F(\mathbb{R}^m-Q_p,n))$. In particular, for $m\geq3$, the lack of a nontrivial fundamental group and torsion in the cohomology  imply that $\F(\mathbb{R}^m-Q_p,n)$ is homotopy equivalent to an $(m-2)$-connected CW complex of dimension $n(m-1)$. The corresponding homotopy model for $m=2$ follows from the results in~\cite{OT}.

\begin{remark}\label{LScatQp}{\em
The above considerations easily give $\cat(\F(\mathbb{R}^m-Q_p,n))=n$ (the cup-length lower bound agrees with the upper bound given by the dimension-by-connectivity ratio). This is of course the analogue of~\cite[Theorem~1.2]{MR2508218}. The relation $\cat(\F(\mathbb{R}^m-Q_p,n))=\cat(\F(\mathbb{R}^m,n))+1$ is closely related to the fact that the parity of $m$ is irrelevant in the fourth instance on the right-hand side of~$(\ref{global})$.
}\end{remark}

\begin{proof}[Proof of Theorem~$\ref{np2}$]
Dimension vs.~connectivity considerations give $$\TC_s(\F(\mathbb{R}^m-Q_p,n))\leq ns$$
(cf.~\cite[Theorem~3.9]{bgrt}, for instance). So it suffices to find a nonzero product of $ns$ factors all of which lie in the kernel of the morphism induced in cohomology by the iterated (thin) diagonal $\F(\mathbb{R}^m-Q_p,n)\to\F(\mathbb{R}^m-Q_p,n)^s$. This can be accomplished for $m$ odd with a calculation identical to the one in the proof of Proposition~\ref{tcl} dealing with the element in~(\ref{longitudmaxima}): this time
$$
\prod^n_{i=1} \left(A^{(1)}_{(p+i)1}+A^{(2)}_{(p+i)1}+\cdots +A^{(s-1)}_{(p+i)1}-(s-1) A^{(s)}_{(p+i)1}\right)^s\neq0.
$$
Yet, in accordance to the last assertion in Remark~\ref{LScatQp}, we offer an argument that works for all $m$. Actually, we work with $\mathbb{Z}_2$-coefficients, where not only can signs be ignored, but the cohomology ring $H^*(\F(\mathbb{R}^m-Q_p,n);\mathbb{Z}_2)$ is really independent of $m$ (except for a shift in gradings which, nevertheless, has no impact on the conclusion we want).

\medskip
The element we are after is
$$
\nu_s=\left(\prod\left(A_{(p+i)1}^{(1)}+A_{(p+i)1}^{(\ell)}\right)\rule{0mm}{6.8mm}\right)\left(\rule{0mm}{6.8mm}\prod_{i=1}^n\left(A_{(p+i)2}^{(1)}+A_{(p+i)2}^{(2)}\right)\right),
$$
where the first product is taken over all pairs $(i,\ell)$ with $1\leq i\leq n$ and $2\leq\ell\leq s$. We show $\nu_s\neq0$ by induction on $s$, noticing that the grounding case $s=2$ is done by Farber, Grant, and Yuzvinsky (see the proof of Theorem~6.1 in~\cite{MR2359030}). In detail, since the cohomology of $\F(\mathbb{R}^m-Q_p,n)$ vanishes above dimension $n(m-1)$, the mod 2 analogue of~(\ref{relazero}) easily yields
$$\nu_2=\sum_{(u_1,\ldots,u_n)\in\{1,2\}^n} A_{(p+n)u_1} A_{(p+n-1)u_2} \cdots A_{(p+1)u_n}\otimes A_{(p+n)(3-u_1)} A_{(p+n-1)(3-u_2)} \cdots A_{(p+1)(3-u_n)},$$
a nonzero element since all the summands in the above expression are distinct basis elements. Note that $\nu_2$ lies in the top nonzero dimension of $H^*(\F(\mathbb{R}^m-Q_p,n))^{\otimes s}$ so that $\nu_2A_{(p+i)1}^{(1)}$ vanishes for all $i=1,\ldots,n$. In particular, for $s\geq2$,
$$
\nu_{s+1}=\nu_s\,\prod_{i=1}^n\left(A_{(p+i)1}^{(1)}+A_{(p+i)1}^{(s+1)}\right)=\nu_s\,A_{(p+n)1}^{(s+1)}\cdots A_{(p+1)1}^{(s+1)}
$$
which is a nonzero element in
$$H^*\!\left(\F(\mathbb{R}^m-Q_p,n)^{s+1};\mathbb{Z}_2\right)=H^*(\F(\mathbb{R}^m-Q_p,n)^s;\mathbb{Z}_2)\otimes H^*(\F(\mathbb{R}^m-Q_p,n);\mathbb{Z}_2)$$
by induction.
\end{proof}

 \end{document}